\documentclass[11pt,a4paper]{amsart}
\usepackage[T1]{fontenc}
\usepackage{lmodern}
\usepackage{amsmath,amssymb,mathtools}
\usepackage[margin=28mm]{geometry}
\usepackage{microtype}
\usepackage[hypertexnames=false]{hyperref}
\hypersetup{colorlinks=true,linkcolor=blue,citecolor=blue,urlcolor=blue,
 pdftitle={A sharp pinching theorem for closed minimal Legendrian submanifolds}}
\allowdisplaybreaks[2]
\newtheorem{theorem}{Theorem}[section]
\newtheorem{lemma}[theorem]{Lemma}
\newtheorem{proposition}[theorem]{Proposition}
\theoremstyle{remark}
\newtheorem{remark}[theorem]{Remark}
\newcommand{\R}{\mathbb R}
\newcommand{\C}{\mathbb C}
\newcommand{\Sph}{\mathbb S}
\DeclareMathOperator{\Ric}{Ric}
\DeclareMathOperator{\tr}{tr}
\DeclareMathOperator{\diver}{div}
\DeclareMathOperator{\Sym}{Sym}

\newcommand{\cn}{c_n}

\newcommand{\dd}{\mathrm d}
\title{A sharp pinching theorem for closed minimal Legendrian submanifolds}
\author{Yong Luo}
\address{Yong Luo, Mathematical Science Research Center of Mathematics,
Chongqing University of Technology, Chongqing 400054, P. R. China}
\email{yongluo-math@cqut.edu.cn}
\author{Linlin Sun}
\address{Linlin Sun, School of Mathematics and Computational Science and
Hunan Research Center of the Basic Discipline Fundamental Algorithmic Theory
and Novel Computational Methods, Xiangtan University, Xiangtan 411105,
P. R. China}
\email{sunll@xtu.edu.cn}
\author{Jiabin Yin}
\address{Jiabin Yin, School of Mathematics and Statistics,
Xinyang Normal University, Xinyang 464000, P. R. China}
\email{jiabinyin@126.com}
\date{}
\subjclass[2020]{53C24, 53C40, 53C42}
\keywords{Minimal Legendrian submanifold, sharp pinching, Simons identity,
 symmetric cubic tensor, Calabi product}
\begin{document}
\begin{abstract}
Let $F:M^n\to\Sph^{2n+1}$ be a minimal Legendrian immersion of a
connected closed manifold, where $n\ge3$, and let $B$ be its second
fundamental form. We prove that if $|B|^2\le(n-1)(n+2)/n$, then, up to
a unitary transformation, $F$ parametrizes a totally geodesic sphere
or is a finite covering of the image of the minimal Calabi product of a
point and a totally geodesic Legendrian sphere. This proves Conjecture 1 of Luo and Sun [Ann. Inst. H. Poincaré C Anal. Non Linéaire 40 (2023), p. 548]. 
\end{abstract}
\maketitle

\section{Introduction}
\label{sec:introduction}

The first gap problem for the squared norm of the second fundamental form of closed minimal submanifolds in a sphere is a
classical problem in submanifold geometry. For an $n$-dimensional minimal
submanifold in the unit sphere $\Sph^{n+p}$, Simons' integral inequality
\cite{Simons} gives the pinching constant $n/(2-1/p)$. Chern, do Carmo and Kobayashi
\cite{ChernDoCarmoKobayashi} classified the nonzero equality cases,
whose images are Clifford minimal hypersurfaces for $p=1$ or the
Veronese surface in $\Sph^4$ for $n=p=2$.

This classification shows that Simons' pinching constant is attained
by non-totally-geodesic models only when $p=1$ or $n=p=2$.
It thus motivates the search for sharper estimates in the remaining
higher-codimension cases, where subsequent improvements have
confirmed that Simons' constant is not optimal.
Li and Li
\cite{LiLi} and Chen and Xu \cite{ChenXu} independently improved this constant to $2n/3$ for codimension $p\ge3$. This pinching constant was further improved recently in several preprints (see \cite{Lei, LiZhao, XuZhao}).

In this paper we consider minimal Legendrian submanifolds in
$\Sph^{2n+1}\subset\C^{n+1}$. An immersion
$F:M^n\to\Sph^{2n+1}$ is Legendrian if
$\langle JF,dF(X)\rangle=0$ for every $X\in TM$, where $J$ is the
standard complex structure on $\C^{n+1}$. Let $B$ denote the second
fundamental form of $F$. Legendrian immersions are horizontal with
respect to the Hopf fibration
$\pi:\Sph^{2n+1}\to\mathbb{CP}^n$. Consequently, $\pi\circ F$ is a
Lagrangian immersion with the same induced metric, when
$\mathbb{CP}^n$ is equipped with the Fubini--Study metric of constant
holomorphic sectional curvature $4$. Moreover, $F$ is minimal if and
only if $\pi\circ F$ is minimal. Conversely, a minimal Lagrangian
immersion in $\mathbb{CP}^n$ admits local horizontal Legendrian lifts.
We shall also use this correspondence when discussing pinching results
for minimal Lagrangian submanifolds (see
\cite[Section~2]{CastroLiUrbano}).

For minimal Legendrian submanifolds, early rigidity theorems used both
intrinsic and extrinsic curvature bounds. Blair and
Ogiue \cite{BlairOgiue} proved that a closed minimal Legendrian
submanifold in $\Sph^{2n+1}$ with
$|B|^2<n^2/(2n-1)$ is totally geodesic. Yamaguchi, Kon and Ikawa
\cite{YamaguchiKonIkawa} obtained a Simons-type formula for
$C$-totally real minimal submanifolds of Sasakian space forms and
further norm pinching results. In the unit sphere, Luo and Sun
\cite[Theorem~4.3]{LuoSun} obtained a stronger pinching result that, for
$n\ge3$, the condition
\[
 |B|^2\le\frac{2(n+1)}{3}
\]
forces a closed minimal Legendrian submanifold to be totally geodesic.

Sectional curvature pinching was studied by Yau \cite{Yau}, Blair and
Ogiue \cite{BlairOgiuePositive}, and Ogiue \cite{Ogiue}. In the
normalization used here, the latter two results give total geodesicity
under the condition $K>(n-2)/(2n-1)$, where $K$ denotes the sectional
curvature. Urbano \cite{Urbano} subsequently proved that every compact
minimal Lagrangian submanifold of $\mathbb{CP}^n$ with $K>0$ is
totally geodesic. Dillen and Vrancken \cite{DillenVrancken} considered
the nonnegative case and proved that a closed minimal Legendrian
submanifold in $\Sph^{2n+1}$ with $0\le K\le1$ has $K\equiv0$ or
$K\equiv1$. Cheng, He and Hu \cite{ChengHeHu} later obtained a complete
classification of minimal Legendrian submanifolds with constant
sectional curvature in Sasakian space forms. In particular, for
$n\ge2$, a minimal Legendrian submanifold of the unit sphere
$\Sph^{2n+1}$ with constant sectional curvature is either totally
geodesic ($K=1$) or flat ($K=0$).

For Ricci curvature, Montiel, Ros and Urbano
\cite{MontielRosUrbano} classified closed minimal Lagrangian
submanifolds of $\mathbb{CP}^n(4)$ satisfying
$\Ric\ge\frac34(n-2)g$. Apart from the totally geodesic case and the
flat minimal torus in dimension two, the immersions are locally
congruent to the standard minimal Lagrangian immersions of
$SU(3)/SO(3)$, $SU(3)$, $SU(6)/Sp(3)$ and $E_6/F_4$, in dimensions
$5$, $8$, $14$ and $26$, respectively.

Rigidity results have also been obtained by controlling the second
fundamental form in individual tangent directions. To formulate
such directional pinching conditions, define
\[
 \Theta(p)=\max\{\,|B(v,v)|:v\in T_pM,\ |v|=1\,\}.
\]
For compact minimal Lagrangian submanifolds of $\mathbb{CP}^n(4)$,
Gauchman \cite{Gauchman} proved that
$\Theta^2<(n+1)/(3n)$ implies total geodesicity. Xia \cite{Xia}
obtained the sharp dimension-independent bound $\Theta^2\le1/2$:
the submanifold is totally geodesic or $\Theta^2\equiv1/2$, with the
same flat and homogeneous models in the nongeodesic case. These
statements apply to minimal Legendrian immersions through their
Lagrangian Hopf projections.

For a closed minimal Legendrian surface in $\Sph^5$, Yamaguchi, Kon and
Miyahara \cite{YamaguchiKonMiyahara} proved that $|B|^2\le2$ implies
that $|B|^2$ is identically zero or two. In the latter case the surface
is a flat minimal Legendrian torus, namely the generalized Clifford
torus (see \cite[Remark~1.5]{LuoSun}). Thus the first gap and its equality case were already known in
dimension two, whereas the corresponding sharp first gap problem
in dimensions $n\ge3$ remained open.

A related variational problem concerns Legendrian submanifolds
that are critical points of the volume functional under contact
deformations \cite{CastroLiUrbano,Iriyeh}. Such submanifolds are called
contact stationary Legendrian (CSL) submanifolds and satisfy
$\diver(JH)=0$, where we use the convention
$H=\tr_g B$ for the mean curvature vector. This notion was motivated by Hamiltonian minimal Lagrangian submanifolds introduced by Oh \cite{Oh}. Every minimal
Legendrian submanifold is CSL. Under the Hopf fibration, the CSL
condition corresponds to Hamiltonian stationarity of the Lagrangian
projection \cite{CastroLiUrbano}. Luo and Sun \cite{LuoSun} studied the
first gap problem in this larger class and obtained optimal rigidity
theorems for nonminimal closed CSL submanifolds in the unit sphere. They proved the following theorem.
\begin{theorem}[\cite{LuoSun}]\label{thm:LuoSun}
Suppose $M^n (n\ge3)$ is a closed contact stationary Legendrian  submanifold of $\mathbb{S}^{2n+1}$ and
\begin{equation*}
|B|^2\leq\dfrac{(n-1)(n+2)}{n}+\dfrac{n^2+3n-2}{2n^2}|H|^2-\dfrac{(n-1)(n-2)|H|\sqrt{4n+|H|^2}}{2n^2}.
\end{equation*}
Then $M$ is either minimal or a Calabi product Legendrian immersion of a totally geodesic Legendrian immersion and a point.
\end{theorem}
\begin{remark}
Recently, Cui and Sun \cite{CuiSun} obtained a topological sphere theorem for Legendrian submanifolds in the unit sphere under the same pinching condition as in Theorem \ref{thm:LuoSun}.
\end{remark}
The minimal Calabi product of a point and a totally geodesic Legendrian
sphere  is called the
Calabi torus, and its second fundamental form (see the appendix of \cite{LuoSun})
satisfies
\[
 |B|^2=\cn:=\frac{(n-1)(n+2)}{n}.
\]
 Motivated by this example and the above theorem, Luo and Sun \cite[Conjecture~1]{LuoSun} conjectured that a
closed minimal Legendrian submanifold in $\Sph^{2n+1}$ satisfying
$|B|^2\le\cn$ must either be totally geodesic or be a Calabi torus.

Luo, Sun and Yin \cite[Theorem~1.1]{LuoSunYin} obtained a
characterization of the Calabi torus in terms of $|B|^2$ and the
directional quantity $\Theta$ defined above. They proved that a closed
minimal Legendrian submanifold satisfying
\[
 |B|^2\le\frac{n+2}{\sqrt n}\Theta
\]
is totally geodesic or a Calabi torus. By using ideas from Lu's paper \cite{Lu}, Wu and Yang \cite{WuYang} proved the same rigidity conclusion under
the condition $|B|^2+\lambda_2\le n+1$, where $\lambda_2$ is the second
largest eigenvalue of the fundamental matrix. Both conditions involve directional information about the second
fundamental form in addition to its norm.

A related sharp rigidity result was obtained by Hu, Sun and Yin
\cite{HuSunYin} for connected closed Lagrangian submanifolds of the
homogeneous nearly K\"ahler sphere $\Sph^6(1)$. Writing $h$ for the
second fundamental form, they proved that $|B|^2\le25/8$ forces the submanifold to be totally
geodesic or congruent to the Berger sphere of
Dillen--Verstraelen--Vrancken. Their proof combines a weighted Simons
integral identity with an integral identity involving covariant
derivatives of the Ricci tensor. We use a similar integral strategy
in the minimal Legendrian setting.

The following theorem proves Conjecture 1 of Luo and Sun \cite{LuoSun} mentioned above.

\begin{theorem}\label{thm:main}
Let $F:M^n\to\Sph^{2n+1}\subset\C^{n+1}$ be a minimal Legendrian
immersion of a connected closed manifold, $n\ge3$. If $|B|^2\le\cn$
everywhere, then, up to a unitary transformation, $F$ parametrizes a
totally geodesic sphere or is a finite covering of the embedded image
\begin{equation}\tag{1.1}\label{eq:1.1}
 (\theta,x)\longmapsto
 \left(\frac{e^{in\theta}}{\sqrt{n+1}},
       \sqrt{\frac n{n+1}}e^{-i\theta}x\right),
 \qquad(\theta,x)\in\Sph^1\times\Sph^{n-1}.
\end{equation}
Here $\Sph^1=\R/(2\pi\mathbb Z)$. For odd $n$ the displayed map is
an embedding; for even $n$ it induces an embedding on the quotient by
$(\theta,x)\sim(\theta+\pi,-x)$.
\end{theorem}

Together with the surface case established by Yamaguchi, Kon and Miyahara \cite{YamaguchiKonMiyahara}, Theorem \ref{thm:main} gives the first gap
for closed minimal Legendrian submanifolds of the unit sphere in all dimensions $n\ge2$.

Our proof is based on an integral argument. The Legendrian condition
allows us to express $B$ as a symmetric cubic tensor
$\sigma(X,Y,Z)=\langle B(X,Y),JZ\rangle$, which is trace-free when
the immersion is minimal. We use the associated Gram tensor
$G_{ij}=\sum_{a,b}\sigma_{iab}\sigma_{jab}$. Combining a weighted
Simons identity \eqref{eq:2.2} with an integral curvature identity \eqref{eq:2.5} for $G$, we reduce
the problem to an algebraic inequality for trace-free symmetric cubics.
The equality case in this inequality determines the cubic tensor of the
Calabi immersion. The Codazzi equation then implies that its
distinguished tangent direction is parallel, and the ambient equations
give the immersion in \eqref{eq:1.1}.

The paper is organized as follows. In Section~\ref{sec:integral} we
establish the integral identity and state the algebraic estimate needed
in the proof. The  algebraic estimate is proved in
Sections~\ref{sec:algebraic}--\ref{sec:upper}.
In Section~\ref{sec:rigidity} we discuss the equality case and complete
the proof of Theorem~\ref{thm:main}.

\section{The cubic tensor and the integral identity}\label{sec:integral}

We express the second fundamental form as a trace-free symmetric cubic
tensor and establish a sharp bound for its Gram tensor. A weighted
Simons identity and an integral curvature identity then reduce the
pinching problem to a homogeneous algebraic inequality of degree six.
The algebraic estimate is stated at the end of this section and proved
in Sections~\ref{sec:algebraic}--\ref{sec:upper}.

Let $g$ be the induced metric, $J$ the standard complex structure, and $D$ the
flat connection on $\mathbb C^{n+1}$. We identify a tangent vector with its image
under $dF$ when it occurs in an ambient inner product. The Legendrian condition
is $\langle JF,X\rangle=0$ for every tangent vector $X$. Taking the exterior
derivative of this vanishing one-form gives $2\langle JX,Y\rangle=0$. Hence
$\langle B(X,Y),JF\rangle=-\langle Y,JX\rangle=0$. Differentiating
$\langle Y,JZ\rangle=0$ shows that
\[
 \sigma(X,Y,Z)=\langle B(X,Y),JZ\rangle
\]
is symmetric in all three entries. We use ordered-index norms and the summation
convention, and put
\[
 S=|\sigma|^2=|B|^2,\qquad G_{ij}=\sigma_{iab}\sigma_{jab},\qquad
 A_i=(\sigma_{ijk})_{j,k},\qquad
 \alpha=\frac{n}{n+2},\qquad \tau_n=\frac{n}{2(n+2)}.
\]
Here closed means compact without boundary, and all integrals use the
Riemannian volume measure. Our conventions are $\Delta=\operatorname{tr}\nabla^2$
and
$R(X,Y)Z=\nabla_X\nabla_YZ-\nabla_Y\nabla_XZ-\nabla_{[X,Y]}Z$.
Thus the sectional curvature of the unit sphere is $1$. The normal derivative
satisfies
\[
 (\nabla^{\perp}_{e_l}B)(e_i,e_j)=\sigma_{ijk;l}Je_k-\sigma_{ijl}JF.
\]
The $Je_k$ components of Codazzi and the minimality condition give
\[
 \sigma_{ijk;l}=\sigma_{ljk;i},\qquad
 \sum_i\sigma_{iik}=\sum_i\sigma_{iik;l}=0.
\]
Together with the symmetry of $\sigma$, these identities say that $\nabla\sigma$
is completely symmetric and trace-free. The Gauss equation is
\[
 K_{ij}=1+\sum_a(\sigma_{iia}\sigma_{jja}-\sigma_{ija}^2),\qquad
 \operatorname{Ric}=(n-1)g-G,\qquad \operatorname{tr}G=S.
\]

\begin{lemma}\label{lem:gram}
Every trace-free symmetric cubic on $\mathbb R^n$, $n\ge3$, satisfies
$0\le G\le\alpha Sg$. If $S>0$ and $G(e,e)=\alpha S$ for a unit vector
$e=e_1$, then
\[
 \sigma_{111}=a,\qquad \sigma_{1\beta\gamma}=-\frac{a}{n-1}\delta_{\beta\gamma},
 \qquad \sigma_{11\beta}=\sigma_{\beta\gamma\delta}=0,\qquad
 S=\frac{n+2}{n-1}a^2,
\]
where Greek indices are greater than one. In this equality case, the largest
Gram eigenvalue is simple. In particular, $S\le c_n$ implies
$\operatorname{Ric}\ge0$.
\end{lemma}

\begin{proof}
Fix an arbitrary unit vector $e_1$. The trace equations give the orthogonal
decomposition
\[
\begin{aligned}
 \sigma_{111}&=a,&\qquad \sigma_{11\beta}&=u_\beta,\\
 \sigma_{1\beta\gamma}&=-\frac{a}{n-1}\delta_{\beta\gamma}+C_{\beta\gamma},\\
 \sigma_{\beta\gamma\delta}
 &=D_{\beta\gamma\delta}
   -\frac{u_\beta\delta_{\gamma\delta}+u_\gamma\delta_{\beta\delta}
          +u_\delta\delta_{\beta\gamma}}{n+1},
\end{aligned}
\]
with $C,D$ trace-free on $e_1^\perp$. Counting the placements of the index $1$,
and using that the squared norm of the three metric insertions of $u$ is
$3(n+1)|u|^2$, gives
\begin{equation}\tag{2.1}\label{eq:2.1}
\begin{aligned}
 S&=\frac{n+2}{n-1}a^2+\frac{3(n+2)}{n+1}|u|^2+3|C|^2+|D|^2,\\
 G_{11}&=\frac{n}{n-1}a^2+2|u|^2+|C|^2.
\end{aligned}
\end{equation}
Consequently,
\[
 \alpha S-G_{11}
 =\frac{n-2}{n+1}|u|^2+\frac{2(n-1)}{n+2}|C|^2+\frac{n}{n+2}|D|^2.
\]
The nonnegative terms prove the bound and show that equality is equivalent to
$u=C=D=0$. At equality, $G_{11}=na^2/(n-1)$ and
$G_{\beta\beta}=2a^2/(n-1)^2$, so the top eigenvalue is simple. Finally,
$\operatorname{Ric}(e,e)=\alpha(c_n-S)+\alpha S-G(e,e)$.
\end{proof}

Set
\[
 Q=|G|^2+\sum_{i,j}|[A_i,A_j]|^2,\qquad V_a=\sigma_{ija}G_{ij}.
\]
At a point, choose a geodesic orthonormal frame and write
\[
 R_{lijk}:=\langle R(e_l,e_i)e_j,e_k\rangle
 =\delta_{ij}\delta_{lk}-\delta_{lj}\delta_{ik}
  +\sigma_{ija}\sigma_{lka}-\sigma_{lja}\sigma_{ika}.
\]
The Codazzi equation and $\nabla_l\sigma_{ljk}=0$ give
\[
\begin{aligned}
 \Delta\sigma_{ijk}
 &=\sum_l(\nabla_l\nabla_i-\nabla_i\nabla_l)\sigma_{ljk}\\
 &=\operatorname{Ric}_{ia}\sigma_{ajk}
   -\sum_{l,a}R_{lija}\sigma_{lak}
   -\sum_{l,a}R_{lika}\sigma_{lja}.
\end{aligned}
\]
For example, the first curvature contraction is
\[
 \sum_{l,a}R_{lija}\sigma_{lak}
 =-\sigma_{ijk}+G_{ka}\sigma_{ija}-\sigma_{iab}\sigma_{jbc}\sigma_{kca}.
\]
The other is obtained by interchanging $j,k$. The Ricci term then gives the
tensor Simons identity \cite[Lemma~2.7]{WuYang}:
\[
 \Delta\sigma_{ijk}
 =(n+1)\sigma_{ijk}-G_{ia}\sigma_{ajk}-G_{ja}\sigma_{aik}
 -G_{ka}\sigma_{ija}+2\sigma_{iab}\sigma_{jbc}\sigma_{kca}.
\]
Contracting with $\sigma$, and using
\[
 \sum_i A_i^2=G,\qquad
 \sum_{i,j}|[A_i,A_j]|^2
 =2|G|^2-2\sum_{i,j}\operatorname{tr}(A_iA_jA_iA_j),
\]
gives
\[
 \tfrac12\Delta S=|\nabla\sigma|^2+(n+1)S-Q.
\]
The coefficient $n+1$ also follows from the usual normal-bundle formula and
$|\nabla^\perp B|^2=|\nabla\sigma|^2+S$. Multiplying the scalar identity by $S$ and integrating by parts gives
\begin{equation}\tag{2.2}\label{eq:2.2}
 0=\int_M\left\{S|\nabla\sigma|^2+\tfrac12|\nabla S|^2+(n+1)S^2-SQ\right\}.
\end{equation}

For any symmetric two-tensor $T$ on a closed manifold, write $t_i$ for its
eigenvalues in a pointwise orthonormal eigenframe. Integration by parts gives
\begin{equation}\tag{2.3}\label{eq:2.3}
 \int_M\sum_{i<j}K_{ij}(t_i-t_j)^2
 =\int_M\left\{|\operatorname{div}T|^2-\nabla_iT_{jk}\nabla_jT_{ik}\right\}.
\end{equation}
Indeed, the right-hand side is
\[
 -\int_M T_{ik}([\nabla_i,\nabla_j]T)_{jk}.
\]
In a pointwise eigenframe the commutator is
\[
 ([\nabla_i,\nabla_j]T)_{ji}=K_{ij}(t_j-t_i).
\]
Pairing the terms indexed by $i,j$ proves the formula. For $T=G$, complete
symmetry gives
\[
 \nabla_iG_{ij}=\sigma_{iab}\sigma_{jab;i}
              =\sigma_{iab}\sigma_{iab;j}=\tfrac12S_{;j}.
\]
If $g_i$ are the Gram eigenvalues, the Gauss equation gives
\begin{equation}\tag{2.4}\label{eq:2.4}
 \sum_{i<j}K_{ij}(g_i-g_j)^2
 =n|G|^2-S^2-\tfrac12\sum_a|[G,A_a]|^2-|V|^2.
\end{equation}
To verify the extrinsic terms, use
\[
\begin{aligned}
 \sum_{i<j}(g_i-g_j)^2\sigma_{iia}\sigma_{jja}
 &=-\left(\sum_i g_i\sigma_{iia}\right)^2,\\
 \sum_{i<j,a}(g_i-g_j)^2\sigma_{ija}^2
 &=\tfrac12\sum_a|[G,A_a]|^2.
\end{aligned}
\]
The first identity uses $\sum_i\sigma_{iia}=0$.
Equations \eqref{eq:2.3} and \eqref{eq:2.4} imply
\begin{equation}\tag{2.5}\label{eq:2.5}
 0=\int_M\left\{
 \nabla_iG_{jk}\nabla_jG_{ik}-\tfrac14|\nabla S|^2+n|G|^2-S^2
 -\tfrac12\sum_a|[G,A_a]|^2-|V|^2\right\}.
\end{equation}
For $t>0$, define
\begin{equation}\tag{2.6}\label{eq:2.6}
\begin{aligned}
 \mathcal D_t&=tS|\nabla\sigma|^2+\nabla_iG_{jk}\nabla_jG_{ik}
       +\frac{2t-1}{4}|\nabla S|^2,\\
 \mathcal B_t&=\frac{n}{c_n}S|G|^2+\frac{t(n+1)-1}{c_n}S^3-tSQ
       -\tfrac12\sum_a|[G,A_a]|^2-|V|^2.
\end{aligned}
\end{equation}
Adding $t$ times \eqref{eq:2.2} to \eqref{eq:2.5}, and separating the norm
deficit, gives
\begin{equation}\tag{2.7}\label{eq:2.7}
 0=\int_M\left\{
 \mathcal D_t+\mathcal B_t+\left(1-\frac{S}{c_n}\right)
 \bigl(n|G|^2+(t(n+1)-1)S^2\bigr)\right\}.
\end{equation}
Since $n|G|^2\ge S^2$,
\begin{equation}\tag{2.8}\label{eq:2.8}
 n|G|^2+(t(n+1)-1)S^2\ge t(n+1)S^2.
\end{equation}

\begin{lemma}\label{lem:derivative}
For $t\ge\tau_n$, one has $\mathcal D_t\ge0$.
\end{lemma}

\begin{proof}
Write $T_{ijkl}=\sigma_{ijk;l}$, and set, within this proof,
\[
 E_{ijk}=\sigma_{iab}T_{abjk},\qquad
 \Gamma_i=\sigma_{abc}T_{abci},\qquad
 E^s_{ijk}=(E_{ijk}+E_{jik}+E_{kij})/3.
\]
The last two indices of $E$ are symmetric, and trace-freeness gives
$\sum_j E_{ijj}=0$ and $\sum_j E_{jij}=\Gamma_i$.
Differentiation of $G$ and expansion of the three summands in $E^s$ yield
\[
\begin{aligned}
 \nabla_iG_{jk}&=E_{jik}+E_{kij},\\
 \nabla_iG_{jk}\nabla_jG_{ik}&=|E|^2+3E_{ijk}E_{jik},\\
 |E^s|^2&=\tfrac13\bigl(|E|^2+2E_{ijk}E_{jik}\bigr).
\end{aligned}
\]
Consequently,
\[
 \nabla_iG_{jk}\nabla_jG_{ik}=-\tfrac12|E|^2+\tfrac92|E^s|^2,
 \qquad \operatorname{tr}E^s=\tfrac23\Gamma.
\]
For $(\delta\Gamma)_{ijk}=\delta_{ij}\Gamma_k+\delta_{ik}\Gamma_j
+\delta_{jk}\Gamma_i$,
\[
 |\delta\Gamma|^2=3(n+2)|\Gamma|^2,\qquad
 |E^s|^2=\left|E^s-\frac{2}{3(n+2)}\delta\Gamma\right|^2
             +\frac{4}{3(n+2)}|\Gamma|^2.
\]
Using $\nabla S=2\Gamma$ yields
\[
 \mathcal D_{\tau_n}
 =\tfrac12\bigl(\alpha S|T|^2-|E|^2\bigr)
  +\tfrac92\left|E^s-\frac{2}{3(n+2)}\delta\Gamma\right|^2
  +\frac{4}{n+2}|\Gamma|^2.
\]
Order the Gram eigenvalues as $g_1\ge\cdots\ge g_n\ge0$.
In a corresponding orthonormal eigenframe, the matrices $A_i/\sqrt{g_i}$
with $g_i>0$ are orthonormal; if $g_i=0$, then $A_i=0$.
For a symmetric matrix $U$, Bessel's inequality gives
\[
\begin{aligned}
 \sum_i\langle A_i,U\rangle^2
 &=\sum_{g_i>0}g_i\left\langle\frac{A_i}{\sqrt{g_i}},U\right\rangle^2\\
 &\le g_1\sum_{g_i>0}\left\langle\frac{A_i}{\sqrt{g_i}},U\right\rangle^2
 \le g_1|U|^2\le\alpha S|U|^2.
\end{aligned}
\]
Apply this to $U_{ab}=T_{abjk}$ and sum over $j,k$ to obtain
$|E|^2\le\alpha S|T|^2$. Finally,
\[
 \mathcal D_t-\mathcal D_{\tau_n}=(t-\tau_n)\left(S|T|^2+\tfrac12|\nabla S|^2\right)\ge0.
\]
\end{proof}

The following algebraic estimate supplies the remaining sign and its
equality implication.

\begin{proposition}\label{prop:sextic}
Let $\sigma$ be a trace-free symmetric cubic on $\mathbb R^n$, $n\ge3$.
For $n\ge4$, put $t=\tau_n$; for $n=3$, put $t=8/25$. Then
$\mathcal B_t\ge0$. If $S>0$ and $\mathcal B_t=0$, then
$\lambda_{\max}(G)=\alpha S$. In dimension three there is the identity
\begin{equation}\tag{2.9}\label{eq:2.9}
 \mathcal B_{8/25}=9\det\left(\tfrac35Sg-G\right).
\end{equation}
The determinant is taken after raising one index with the metric.
\end{proposition}

The proof occupies Sections~\ref{sec:algebraic}--\ref{sec:upper}. All estimates there concern a fixed cubic
tensor and do not require it to arise from an immersion.

\section{Algebraic reduction and the lower spectral range}
\label{sec:algebraic}

Orthogonal projections of symmetric tensors give a lower bound for the
algebraic term in terms of Gram spectral moments and two nonnegative
remainders. This proves strict positivity when
$\lambda_{\max}(G)\le S/2$. The same bound, with the remainders
retained, will be used for the other spectral ranges.

\subsection{Projection estimates}

We first normalize the cubic and derive the projection estimates.

The tensor identities in this section hold for $n\ge3$; the subsequent scalar
estimates will assume $n\ge4$. Let $\sigma$ be a nonzero trace-free symmetric
cubic on $\mathbb R^n$. Throughout Sections~\ref{sec:algebraic}--\ref{sec:upper}, rescale this
fixed tensor so that $S=1$, and write
\[
 \ell=\lambda_{\max}(G),\qquad p_2=\operatorname{tr}G^2,\qquad
 p_3=\operatorname{tr}G^3,\qquad r=|V|^2,\qquad \tau=\tau_n.
\]
Every term of $\mathcal B_t$ has degree six in $\sigma$, so the general case is
recovered by homogeneity.

The normalized Calabi cubic is the equality tensor of Lemma~\ref{lem:gram}:
\begin{equation}\tag{3.1}\label{eq:3.1}
 (\sigma_0)_{111}=\sqrt{\frac{n-1}{n+2}},\qquad
 (\sigma_0)_{1\beta\gamma}
 =-\frac{\delta_{\beta\gamma}}{\sqrt{(n-1)(n+2)}},\qquad
 (\sigma_0)_{11\beta}=(\sigma_0)_{\beta\gamma\delta}=0.
\end{equation}
Its Gram spectrum is
\[
 \left(\alpha,\frac{2}{(n-1)(n+2)},\ldots,
 \frac{2}{(n-1)(n+2)}\right).
\]
For a largest Gram direction, Lemma~\ref{lem:gram} gives
\begin{equation}\tag{3.2}\label{eq:3.2}
 \varepsilon:=\alpha-\ell
 =\frac{n-2}{n+1}|u|^2
  +\frac{2(n-1)}{n+2}|C|^2+\frac{n}{n+2}|D|^2.
\end{equation}
Thus $\varepsilon=0$ characterizes the normalized Calabi cubic up to an
orthogonal change of basis. The letters $C,D$ in \eqref{eq:3.2} refer
to the decomposition in Lemma~\ref{lem:gram}.

For a symmetric $j$-tensor $U$, set
$(\operatorname{tr}U)_{i_3\cdots i_j}=\sum_aU_{aai_3\cdots i_j}$.
Write $\operatorname{Sym}^j_0(\mathbb R^n)$ for the kernel of this trace, and
let $h_j$ be the orthogonal projection onto that space. In a Gram eigenbasis,
the sum of the actions of $G$ on the three slots has components
$(g_i+g_j+g_k)\sigma_{ijk}$. Its trace-free part and its component along
$\sigma$ will reorganize the commutator term. Define
\[
 N_{ijk}=G_{ia}\sigma_{ajk}+G_{ja}\sigma_{aik}+G_{ka}\sigma_{ija},
 \qquad W=h_3N-3p_2\sigma.
\]
Put
\begin{align*}
 A&=-\frac{n(n^2-2n-4)}{(n-1)(n+2)(n+4)},&
 C&=\frac{n(n^3+3n^2-12n-12)}{2(n-1)(n+2)^2(n+4)},\\
 k(\ell)&=\frac{12(n+2)^2}{(n+4)\{n+4+3(n+2)\ell\}},&
 z(\ell)&=\tau k(\ell)-\alpha,\\
 L&=\frac{2n\ell^2-4\ell+1}{n-2}.&&
\end{align*}
The constants $A,C$ are, equivalently,
$n/c_n-\tau\,4(n+2)/(n+4)$ and
$(\tau(n+1)-1)/c_n-\tau\,4/((n+2)(n+4))$.

\begin{lemma}\label{lem:projection}
One has $z(\ell)>0$ and
\begin{equation}\tag{3.3}\label{eq:3.3}
 \mathcal B_\tau\ge\frac32p_2^2-\frac32p_3+Ap_2+C+z(\ell)r
 +\frac16|W|^2.
\end{equation}
For $\ell>1/2$, one also has
\begin{equation}\tag{3.4}\label{eq:3.4}
 \mathcal B_\tau\ge\frac32p_2^2-\frac32p_3+Ap_2+C
 +z(\ell)(2\ell-1)(L-p_2)+\frac16|W|^2.
\end{equation}
\end{lemma}

\begin{proof}
To estimate $Q$, take the completely symmetric part of the contracted product
$\sigma_{ija}\sigma_{kla}$:
\[
 \mathcal A_{ijkl}
 =\frac13(\sigma_{ija}\sigma_{kla}+\sigma_{ika}\sigma_{jla}
       +\sigma_{ila}\sigma_{jka}),\qquad K=h_4\mathcal A.
\]
The metric insertions are
\begin{align*}
 (\delta G)_{ijkl}
 &=\delta_{ij}G_{kl}+\delta_{ik}G_{jl}+\delta_{il}G_{jk}
   +\delta_{jk}G_{il}+\delta_{jl}G_{ik}+\delta_{kl}G_{ij},\\
 (\delta\delta)_{ijkl}
 &=\delta_{ij}\delta_{kl}+\delta_{ik}\delta_{jl}
   +\delta_{il}\delta_{jk}.
\end{align*}
For a symmetric two-tensor $H$ and a symmetric four-tensor $U$, complete
symmetry gives
\[
 \langle\delta H,U\rangle=6\langle H,\operatorname{tr}U\rangle,
 \qquad
 \langle\delta\delta,U\rangle=3\operatorname{tr}^2U,
\]
where $\operatorname{tr}^2U=\sum_{i,j}U_{iijj}$. Thus the metric insertions
are orthogonal to trace-free symmetric four-tensors. Taking traces gives
\[
 \operatorname{tr}\mathcal A=\frac23G,\qquad
 \operatorname{tr}(\delta G)=(n+4)G+g,\qquad
 \operatorname{tr}(\delta\delta)=(n+2)g.
\]
It follows that the orthogonal trace part is
\[
 \mathcal A-K=\frac{2}{3(n+4)}\delta G
 -\frac{2}{3(n+4)(n+2)}\delta\delta.
\]
Its squared norm can be computed by pairing once with $\mathcal A$:
\[
 |\mathcal A-K|^2
 =\langle\mathcal A,\mathcal A-K\rangle
 =\frac{8}{3(n+4)}p_2-\frac{4}{3(n+2)(n+4)}.
\]
If $H_0=\sum_{i,j}\operatorname{tr}(A_iA_jA_iA_j)$, then
$Q=3p_2-2H_0$ and $|\mathcal A|^2=(p_2+2H_0)/3$. The first identity
follows from $\sum_iA_i^2=G$; the second follows by expanding the three
squared terms and six cross terms in $|\mathcal A|^2$. Subtracting the trace
part yields
\begin{equation}\tag{3.5}\label{eq:3.5}
 3|K|^2=\frac{4(n+2)}{n+4}p_2
       +\frac{4}{(n+2)(n+4)}-Q,
 \qquad K_{iabc}\sigma_{abc}=\frac{n+2}{n+4}V_i.
\end{equation}
For the second identity, first note that
$\mathcal A_{iabc}\sigma_{abc}=V_i$. In $\delta G$, exactly three metric
insertions survive the contraction, each contributing $V_i$.

For a unit vector $e$, put $A_e=\sigma(e,\cdot,\cdot)$ and
$T_e=\operatorname{Sym}(e^\flat\otimes\sigma)$, where the four placements
are averaged. The four diagonal terms and twelve cross terms in $|T_e|^2$
give
\[
 |T_e|^2=\frac14+\frac34|A_e|^2,\qquad
 \operatorname{tr}T_e=\frac12A_e,\qquad
 h_4T_e=T_e-\frac{\delta A_e}{2(n+4)}.
\]
The same trace pairing then gives
\begin{align*}
 |h_4T_e|^2
 &=|T_e|^2-\frac{\langle\delta A_e,T_e\rangle}{2(n+4)}\\
 &=\frac14+\frac34|A_e|^2-\frac{3|A_e|^2}{2(n+4)}
 =\frac14+\frac{3(n+2)}{4(n+4)}G(e,e).
\end{align*}
When $V\ne0$, choose $e=V^\sharp/|V|$. Since
$\langle K,h_4T_e\rangle=(n+2)|V|/(n+4)$ and $G(e,e)\le\ell$,
Cauchy's inequality gives
\[
 \frac{(n+2)^2}{(n+4)^2}r
 =\langle K,h_4T_e\rangle^2
 \le |K|^2\left\{\frac14+\frac{3(n+2)}{4(n+4)}\ell\right\}.
\]
Eliminating $|K|^2$ by \eqref{eq:3.5} yields
\begin{equation}\tag{3.6}\label{eq:3.6}
 Q\le\frac{4(n+2)}{n+4}p_2
     +\frac{4}{(n+2)(n+4)}-k(\ell)r.
\end{equation}
For $V=0$ this follows from $|K|^2\ge0$.

Next, $\operatorname{tr}N=2V$ and $\langle N,\sigma\rangle=3p_2$ imply
\begin{equation}\tag{3.7}\label{eq:3.7}
 |N|^2=9p_2^2+\frac{12}{n+2}r+|W|^2.
\end{equation}
Indeed, the trace part of $N$ is $2\delta V/(n+2)$, with squared norm
$12r/(n+2)$, and $W\perp\sigma$. In a Gram eigenframe,
\[
 |N|^2=3p_3+6\sum_{i,j,a}g_ig_j\sigma_{ija}^2,
 \qquad
 \frac12\sum_a|[G,A_a]|^2
 =p_3-\sum_{i,j,a}g_ig_j\sigma_{ija}^2.
\]
Thus \eqref{eq:2.6} becomes
\[
 \mathcal B_\tau=\frac{n}{c_n}p_2+\frac{\tau(n+1)-1}{c_n}
 -\tau Q-\frac32p_3+\frac16|N|^2-r.
\]
Substituting \eqref{eq:3.6} and \eqref{eq:3.7} gives \eqref{eq:3.3}; the
coefficient of $r$ is $\tau k(\ell)+2/(n+2)-1=z(\ell)$. Also $k$ is
decreasing and
\begin{equation}\tag{3.8}\label{eq:3.8}
 z(\ell)\ge z(\alpha)
 =\frac{n(n^2+2n+4)}{2(n+1)(n+2)(n+4)}>0.
\end{equation}

To bound $r$ from below, split the cubic relative to a largest-eigenvalue
direction. The quantity $2\ell-1$ arising in this split is positive precisely
in the range now under consideration, $\ell>1/2$. Choose a Gram eigenframe
with $g_1=\ell$ and put
\begin{gather*}
 m=n-1,\qquad \mu=\frac{1-\ell}{m},\qquad d=\ell-\mu,\\
 Z=G|_{e_1^\perp}-\mu I,\qquad \eta=|Z|^2,\qquad
 \mathcal K=\frac{m}{m-1}d^2-\eta=L-p_2.
\end{gather*}
Since $p_2\le\ell^2+(1-\ell)^2$,
\[
 \mathcal K\ge(2\ell-1)
 \left(1+\frac{2\ell-1}{n-2}\right)>0.
\]
Let $a_i=\sigma_{i11}$ and
$T_{i\beta\gamma}=\sigma_{i\beta\gamma}+a_i\delta_{\beta\gamma}/m$.
Each $T_i$ is trace-free on the transverse indices. Here
$|a_\bullet|^2=\sum_i a_i^2$ and
$|T|^2=\sum_{i,\beta,\gamma}T_{i\beta\gamma}^2$. Separating the entries
according to the number of indices equal to $1$ gives
\[
 S=2\ell-|a_\bullet|^2+\sum_{i,\beta,\gamma}\sigma_{i\beta\gamma}^2,
 \qquad
 |T|^2=\sum_{i,\beta,\gamma}\sigma_{i\beta\gamma}^2
       -\frac{|a_\bullet|^2}{m}.
\]
Since $S=1$ and $G$ is diagonal in the chosen frame, it follows that
\[
 2\ell-1=\frac{m-1}{m}|a_\bullet|^2-|T|^2,
 \qquad V_i=da_i+\langle T_i,Z\rangle.
\]
With $\rho=(m-1)/(md)$, the required square completion is
\[
 r-(2\ell-1)\mathcal K
 =\mathcal K\sum_i|T_i+\rho a_iZ|^2
  +\sum_i\langle T_i+\rho a_iZ,Z\rangle^2\ge0.
\]
Expanding the right-hand side,
$\mathcal K+\eta=md^2/(m-1)$ gives the cross coefficient
$2\rho(\mathcal K+\eta)=2d$ and the $|a_\bullet|^2$ coefficient
$\rho^2\eta(\mathcal K+\eta)=(m-1)\eta/m$.
Together with \eqref{eq:3.8}, this proves \eqref{eq:3.4}.
\end{proof}

\subsection{The lower spectral range}

When $\ell\le1/2$, a third-moment bound is sufficient; the two
nonnegative remainders in \eqref{eq:3.3} may be discarded. The resulting
scalar quadratic has a strictly positive lower bound.

\begin{proposition}\label{prop:lower}
For $n\ge4$ and $\ell\le1/2$, one has $\mathcal B_\tau>0$.
\end{proposition}

\begin{proof}
Since $p_2=\sum_i g_i^2\le\frac12\sum_i g_i=1/2$, set
$x=1/2-p_2\ge0$. For every real $a$,
\[
 0\le\sum_i(1/2-g_i)(g_i-a)^2
 =\frac12p_2-p_3-2ax+\frac{n-2}{2}a^2.
\]
Taking $a=2x/(n-2)$ yields $p_3\le p_2/2-2x^2/(n-2)$. After discarding
the nonnegative $r$ and $W$ terms in \eqref{eq:3.3}, we obtain
\begin{gather*}
 \mathcal B_\tau\ge D_n-L_nx+\frac{3n}{2(n-2)}x^2,\\
 D_n=\frac A2+C
 =\frac{n(n-2)(3n+2)}{2(n-1)(n+2)^2(n+4)},\qquad
 L_n=A+\frac34.
\end{gather*}
Division in the dimension parameter gives
\begin{equation}\tag{3.9}\label{eq:3.9}
 A=-1+\frac7{n+4}+\frac{6-n}{(n-1)(n+2)(n+4)}.
\end{equation}
Consequently,
\begin{align*}
 \frac3n-L_n
 &=\frac1{n+4}\left\{
 \frac{(n-6)^2+12}{4n}+\frac{n-6}{(n-1)(n+2)}\right\}\\
 &\ge\frac1{n+4}\left\{\frac3n-\frac2{(n-1)(n+2)}\right\}>0.
\end{align*}
Here $n\ge4$ and $(n-1)(n+2)>n$. The constant term satisfies
\begin{equation}\tag{3.10}\label{eq:3.10}
 D_n-\frac{3(n-2)}{4n(n+2)}
 =\frac{(n-2)\{(n-4)(3n^2+n-2)+16\}}
        {4n(n-1)(n+2)^2(n+4)}>0.
\end{equation}
Using $x\ge0$ before replacing the linear coefficient, we get
\[
 \mathcal B_\tau>
 \frac{3n}{2(n-2)}\left(x-\frac{n-2}{n^2}\right)^2
 +\frac{3(n-2)(n^2-2n-4)}{4n^3(n+2)}>0.
\]
\end{proof}

\section{The intermediate spectral range}
\label{sec:intermediate}

For $1/2\le\ell\le2/3$, the projection bound reduces to a quadratic
in the variance of the transverse Gram eigenvalues. We prove its
nonnegativity by estimating the three Bernstein coefficients. One
endpoint estimate also extends to $\ell=\alpha$ in dimensions five
and six, settling those cases.

Assume $n\ge4$ and $\ell\ge1/2$, and put
\[
 s=1-\ell,\qquad
 \bar p_2=\ell^2+\frac{s^2}{n-1},\qquad
 \bar p_3=\ell^3+\frac{s^3}{(n-1)^2},\qquad
 \eta=p_2-\bar p_2,\qquad
 q=\frac{n-2}{n-1}s^2.
\]
Then $0\le\eta\le q$ by the Cauchy--Schwarz inequality and $\sum_{i>1}g_i^2\le(\sum_{i>1}g_i)^2$. Since $\ell\le\alpha<1$, one has $q>0$. For $m$ nonnegative numbers $x_a$ with sum $s$,
\[
 (m+1)s\sum_a x_a^2-s^3-m\sum_a x_a^3
 =\sum_{a<b}(x_a-x_b)^2(s-x_a-x_b)\ge0.
\]
Taking $m=n-1$ and using the transverse Gram eigenvalues yields
\begin{equation}
 p_3\le\bar p_3+\frac{n}{n-1}s\eta.
 \tag{4.1}\label{eq:4.1}
\end{equation}
Dropping $|W|^2$ in \eqref{eq:3.4}, we obtain
\begin{equation}
 \mathcal B_\tau\ge b_0(\ell)+M_0(\ell)\eta+\frac32\eta^2,
 \qquad 0\le\eta\le q,
 \tag{4.2}\label{eq:4.2}
\end{equation}
where
\begin{equation}
\begin{aligned}
 b_0&=\frac32(\bar p_2^2-\bar p_3)+A\bar p_2+C
       +(2\ell-1)z(\ell)(L-\bar p_2),\\
 M_0&=3\bar p_2+A-(2\ell-1)z(\ell)-\frac{3n}{2(n-1)}s.
\end{aligned}
\tag{4.3}\label{eq:4.3}
\end{equation}
At $\ell=1/2$, the same inequality follows directly from \eqref{eq:3.3}. Denote the value of the quadratic at its right endpoint by
\[
 b_2(\ell)=b_0+M_0q+\frac32q^2.
\]
The endpoint values correspond to equal transverse eigenvalues and
maximal transverse variance. We estimate the quadratic on the full
interval $[0,q]$, whether or not each value is realized by a cubic.
Set $v=\eta/q$ and
\[
 b_1=\frac12\left(b_0+b_2-\frac32q^2\right).
\]
Since $q>0$, the identity
\begin{equation}
 b_0+M_0\eta+\frac32\eta^2
 =b_0(1-v)^2+2b_1v(1-v)+b_2v^2,
 \qquad 0\le v\le1,
 \tag{4.4}\label{eq:4.4}
\end{equation}
reduces the problem to the signs of $b_0,b_1,b_2$. The endpoint bounds
in the next two lemmas also give $b_1>0$.

\subsection{Endpoint estimates}

We estimate the endpoint coefficients $b_0$ and $b_2$ first.

\begin{lemma}\label{lem:interpolation}
For $n\ge4$ and $1/2\le\ell\le2/3$,
\begin{equation}
 b_0(\ell)\ge\frac12(\alpha-\ell)(1-\ell).
 \tag{4.5}\label{eq:4.5}
\end{equation}
For $n=4,5,6$, the same estimate holds on $[1/2,\alpha]$.
\end{lemma}

\begin{proof}
Set $\varepsilon=\alpha-\ell$, $d=(n\ell-1)/(n-1)$ and $\Lambda=n+4+3(n+2)\ell$. The identities
\[
 \bar p_2=\frac{1+(n-1)d^2}{n},\qquad
 \bar p_2^2-\bar p_3=-\ell(1-\ell)d^2,\qquad
 L-\bar p_2=\frac{n-1}{n-2}d^2
\]
reduce \eqref{eq:4.3} to
\begin{equation}
 b_0=\frac An+C+d^2\left\{
       \frac{n-1}{n}A-\frac32\ell(1-\ell)
       +\frac{n-1}{n-2}(2\ell-1)z(\ell)\right\}.
 \tag{4.6}\label{eq:4.6}
\end{equation}
In particular, $b_0(\alpha)=0$. For $\ell<\alpha$, define $\varphi_n(\ell)=b_0(\ell)/(\alpha-\ell)$. The following computation also establishes its continuous extension at $\alpha$. Write
\[
 d_0=\frac{(n-2)(n+1)}{(n-1)(n+2)},\qquad
 c=\frac{3n(n-1)(n+2)(5n+14)}{8(n-2)(n+1)^2(n+4)},\qquad
 D=\frac{\Lambda}{4(n+1)}.
\]
A direct calculation gives
\begin{equation}
\begin{aligned}
 \varphi_n={}&\frac{n(n-2)}{2(n-1)(n+2)}
  +\frac{n}{2(n+1)}d
  +\left(2+\frac1{n-2}-\frac3{2n}\right)d^2\\
 &-\frac{3(n-1)}{2n}d^3-\frac{cd^2}{D}.
\end{aligned}
\tag{4.7}\label{eq:4.7}
\end{equation}
To obtain this expression, evaluate the constant term and the coefficient of $d^2$ in \eqref{eq:4.6} at $\ell=\alpha$:
\[
 \frac An+C=\frac{(n-2)^2(n+1)}{2(n-1)(n+2)^2},
\]
\[
 \left\{\frac{n-1}{n}A-\frac32\ell(1-\ell)
 +\frac{n-1}{n-2}(2\ell-1)z(\ell)\right\}_{\ell=\alpha}
 =-\frac{n-1}{2(n+1)}.
\]
These values satisfy $b_0(\alpha)=0$ and $d_0-d=n\varepsilon/(n-1)$. The only rational subtraction is
\[
 \frac{(2\ell-1)z(\ell)-(2\alpha-1)z(\alpha)}{\varepsilon}
 =2\alpha-\frac{3n(n+2)(5n+14)}{2(n+1)(n+4)\Lambda}.
\]
Indeed, in the two linear fractions defining $z$, the denominator at $\alpha$ is $4(n+1)$ and the cross-multiplied numerator is linear in $\ell-\alpha$. Subtracting $b_0(\alpha)$ from \eqref{eq:4.6} now gives
\[
 \varphi_n=\frac{n}{2(n+1)}(d_0+d)
 +d^2\left\{-\frac32(2\alpha-1)+\frac32\varepsilon
 +\frac{2\alpha(n-1)}{n-2}-\frac cD\right\}.
\]
Substitute $\varepsilon=(n-1)(d_0-d)/n$ to obtain \eqref{eq:4.7}.

We compare $\varphi_n$ with its affine endpoint interpolant. Put
\[
 a=\frac{n+6}{4n},\qquad
 h=\frac{3(n-1)(n+2)}{4n(n+1)},\qquad
 D(d)=a+hd.
\]
Here $a,h,c$ are positive. Regard \eqref{eq:4.7} as the function $\Phi(d)=\varphi_n((1+(n-1)d)/n)$. For $u<v$, let
\[
 (\mathcal Lf)(d)=\frac{v-d}{v-u}f(u)+\frac{d-u}{v-u}f(v)
\]
be the affine interpolant at $u,v$. The three elementary differences needed for \eqref{eq:4.7} are
\begin{align*}
 d^2-\mathcal L(d^2)&=-(d-u)(v-d),\\
 d^3-\mathcal L(d^3)&=-(d-u)(v-d)(d+u+v),\\
 \frac{d^2}{D(d)}-\mathcal L\left(\frac{d^2}{D(d)}\right)
 &=-\frac{a^2(d-u)(v-d)}{D(d)D(u)D(v)}.
\end{align*}
For the last identity, divide $d^2/D(d)=d/h-a/h^2+a^2/(h^2D(d))$; the affine terms cancel, and direct subtraction of the remaining three reciprocals gives the displayed expression. Since the constant and linear terms in \eqref{eq:4.7} also cancel, we have the exact remainder formula
\begin{equation}
 \Phi(d)-\mathcal L\Phi(d)=(d-u)(v-d)H_{u,v}(d),
 \tag{4.8}\label{eq:4.8}
\end{equation}
where
\[
 H_{u,v}(d)=-2-\frac1{n-2}+\frac3{2n}
 +\frac{3(n-1)}{2n}(d+u+v)+\frac{ca^2}{D(d)D(u)D(v)}.
\]
On $u\le d\le v$, use only $d\ge u$ and $D(d)\le D(v)$:
\begin{equation}
 H_{u,v}(d)\ge-2-\frac1{n-2}+\frac3{2n}
 +\frac{3(n-1)}{2n}(2u+v)+\frac{ca^2}{D(u)D(v)^2}
 =:H_*.
 \tag{4.9}\label{eq:4.9}
\end{equation}
Choose $u=(n-2)/(2(n-1))$, corresponding to $\ell=1/2$. For $n=4,5,6$ take $v=d_0$, corresponding to $\ell=\alpha$. Here
\[
 D(u)=\frac{5n+14}{8(n+1)},\qquad D(v)=1,
\]
\[
 H_*=\frac{n^3-7n^2+12}{n(n-2)(n+2)}
 +\frac{3(n-1)(n+2)(n+6)^2}{16n(n-2)(n+1)(n+4)}.
\]
Substitution gives the positive values
\[
 \begin{array}{c|ccc}
 n&4&5&6\\\hline
 H_*&39/128&3193/7560&29/56
 \end{array}.
\]
For $n\ge7$ take $v=(2n-3)/(3(n-1))$, corresponding to $\ell=2/3$. Then $D(v)=(3n+8)/(4(n+1))$, and cancellation in \eqref{eq:4.9} gives
\[
 H_*=\frac12-\frac3n-\frac1{n-2}
 +3\frac{(n^2-1)(n+2)}{n(n-2)(n+4)}
 \left(\frac{n+6}{3n+8}\right)^2.
\]
The first fraction in the last term is greater than one, because
\[
 (n^2-1)(n+2)-n(n-2)(n+4)=7n-2>0;
\]
the squared fraction is greater than $1/9$. Consequently,
\[
 H_*>\frac12-\frac12-\frac14+\frac13=\frac1{12}>0.
\]
Thus \eqref{eq:4.8} puts $\Phi$ above its endpoint interpolant on each required interval.

It remains to check the endpoints against $(1-\ell)/2$, which is affine in $d$. At $\ell=1/2$, substitution of $d=(n-2)/(2(n-1))$ and $\Lambda=(5n+14)/2$ into \eqref{eq:4.7} gives
\begin{equation}
 \varphi_n(1/2)-\frac14
 =\frac{(n-4)(n^3+18n^2-8n-16)}{16(n-1)^2(n+2)(n+4)}\ge0.
 \tag{4.10}\label{eq:4.10}
\end{equation}
For $n=4,5,6$, use $d=d_0$ and $D(d_0)=1$ in the same formula:
\[
 \begin{array}{c|ccc}
 n&4&5&6\\\hline
 \varphi_n(\alpha)-1/(n+2)&7/216&481/10976&29/800
 \end{array}.
\]
For $n\ge7$, the right endpoint is $2/3$. Substituting $d=(2n-3)/(3(n-1))$ and $\Lambda=3n+8$ gives
\[
 \varphi_n(2/3)-\frac16
 =\frac{2\{n(n-4)(16n^2+37n-3)+6(27n+32)\}}
 {9(n-2)(n-1)(n+2)(n+4)(3n+8)}>0.
\]
The endpoint interpolation now yields $\varphi_n(\ell)\ge(1-\ell)/2$ throughout the required intervals. Multiplication by $\alpha-\ell\ge0$ proves the lemma, including the Calabi endpoint by continuous extension.
\end{proof}

\begin{lemma}\label{lem:opposite}
For $n\ge4$ and $1/2\le\ell\le2/3$,
\begin{equation}
 b_2(\ell)>\frac{(2\ell-1)^2}{6}.
 \tag{4.11}\label{eq:4.11}
\end{equation}
\end{lemma}

\begin{proof}
At $\eta=q$, the two moments in \eqref{eq:4.2} are $\ell^2+(1-\ell)^2$ and $\ell^3+(1-\ell)^3$. Put $x=2\ell-1\in[0,1/3]$ and $v=1/3$. Using $\ell(1-\ell)=(1-x^2)/4$ gives
\begin{equation}
\begin{aligned}
 b_2&=D_n+x^2\omega(x),\\
 \omega(x)&=\frac38x^2+\frac A2-\frac38
 +z((1+x)/2)\left(1+\frac{x}{n-2}\right),
\end{aligned}
\tag{4.12}\label{eq:4.12}
\end{equation}
where $D_n=A/2+C>0$ is the constant in \eqref{eq:3.10}. To estimate $\omega$, set
\[
 a_* =\frac{5n+14}{3(n+2)},\qquad b_* =\frac{4n}{n+4}.
\]
Then $z((1+x)/2)=b_* /(a_*+x)-\alpha$, and
\[
 n-2-a_* =\frac{(n-4)(3n+7)+2}{3(n+2)}>0.
\]
Subtract the two quadratic terms and the two linear fractions in \eqref{eq:4.12}. This gives directly
\[
 \omega(x)-\omega(v)
 =(v-x)\left\{-\frac38(x+v)
 +\frac{b_*(1-a_* /(n-2))}{(a_*+x)(a_*+v)}
 +\frac{\alpha}{n-2}\right\}.
\]
Since $x\le v$, the braces are bounded below by
\[
 -\frac14+\frac{b_*}{(a_*+v)^2}
 \left(1-\frac{a_*}{n-2}\right)+\frac{\alpha}{n-2}.
\]
The factor in parentheses is positive, and
\[
 \frac{b_*}{(a_*+v)^2}
 =\frac{n}{n+4}\left(1-\frac2{3n+8}\right)^2
 \ge\frac12(9/10)^2>\frac25.
\]
Also $a_*\le17/9$ and $\alpha\ge2/3$, so the braces exceed
\[
 -\frac14+\frac25-\frac4{45(n-2)}
 \ge-\frac14+\frac25-\frac2{45}>0.
\]
Hence $\omega(x)\ge\omega(1/3)$.

To check the resulting endpoint, set $z_2=z(2/3)>0$. Substitution of $x=1/3$ into \eqref{eq:4.12} gives
\begin{equation}
 E_n:=b_2(2/3)-\frac1{54}
 =D_n+\frac A{18}-\frac1{18}+\frac{3n-5}{27(n-2)}z_2.
 \tag{4.13}\label{eq:4.13}
\end{equation}
For $n=4$, the values $D_4=7/108$, $A=-1/9$ and $z_2=7/30$ give $E_4=1/30>0$. For $n\ge5$, discard the following explicitly positive differences:
\begin{align*}
 A-\left(-1+\frac6{n+4}\right)
 &=\frac{n^2+4}{(n-1)(n+2)(n+4)}>0,\\
 z_2-\left(1-\frac8{n+4}\right)
 &=\frac{2(n^2+16n+32)}{(n+2)(n+4)(3n+8)}>0,\\
 \frac{3n-5}{27(n-2)}-\frac19
 &=\frac1{27(n-2)}>0.
\end{align*}
Use the last estimate first, with $z_2>0$, and then the first two in \eqref{eq:4.13}. The constant terms cancel, leaving
\[
\begin{aligned}
 E_n&>D_n-\frac5{9(n+4)}
 >\frac{3(n-2)}{4n(n+2)}-\frac5{9(n+4)}\\
 &=\frac{(n-5)(7n+49)+29}{36n(n+2)(n+4)}>0.
\end{aligned}
\]
Here the second inequality is \eqref{eq:3.10}. Finally, the expression $D_n+x^2\{\omega(1/3)-1/6\}$ is affine in $x^2\in[0,1/9]$. Therefore
\[
 b_2-x^2/6\ge D_n+x^2\{\omega(1/3)-1/6\}
 \ge\min\{D_n,E_n\}>0,
\]
as required.
\end{proof}

\subsection{Positivity in the intermediate range}

The endpoint bounds give positivity of the middle Bernstein coefficient
on $1/2\le\ell\le2/3$.

\begin{proposition}\label{prop:middle}
For $n\ge4$ and $1/2\le\ell\le2/3$, one has $\mathcal B_\tau\ge0$.
\end{proposition}

\begin{proof}
By \eqref{eq:4.4}, it remains to verify the signs of $b_0,b_1,b_2$. Lemmas \ref{lem:interpolation} and \ref{lem:opposite} give $b_0\ge0$ and $b_2>0$. To bound $b_1$, put $x=2\ell-1$ and $\rho=2\alpha-1=(n-2)/(n+2)$. Then
\[
 \alpha-\ell=(\rho-x)/2,\qquad 1-\ell=(1-x)/2,\qquad
 q=\frac{\rho}{1+3\rho}(1-x)^2,\qquad \rho\ge1/3.
\]
The inequality
\[
 \frac\rho{12}-\frac{\rho^2}{(1+3\rho)^2}
 =\frac{\rho(3\rho-1)^2}{12(1+3\rho)^2}\ge0
\]
together with \eqref{eq:4.5} and \eqref{eq:4.11} gives
\begin{align*}
 2b_1&>\frac{(\rho-x)(1-x)}8+\frac{x^2}6
 -\frac{3\rho^2}{2(1+3\rho)^2}(1-x)^4\\
 &\ge\frac{3(\rho-x)(1-x)+4x^2-3\rho(1-x)^4}{24}\\
 &=\frac{3(\rho-1/3)x(1-x)(3-3x+x^2)+x^2(1+4x-x^2)}{24}\ge0.
\end{align*}
The quadratic is nonnegative on its required interval.
\end{proof}

\subsection{The upper range in dimensions five and six}

For $n=5,6$, the bound for $b_0$ extends to the Calabi endpoint.
It therefore suffices to show that the linear coefficient $M_0$ is
positive on the remaining interval.

\begin{proposition}\label{prop:small-upper}
For $n=5,6$ and $2/3\le\ell\le\alpha$, one has $\mathcal B_\tau\ge0$.
\end{proposition}

\begin{proof}
Lemma \ref{lem:interpolation} gives $b_0\ge(\alpha-\ell)(1-\ell)/2$. Since $z'<0$ and
\begin{equation}
 1-z(2/3)=\frac{2(11n^2+40n+32)}{(n+2)(n+4)(3n+8)}>0,
 \tag{4.14}\label{eq:4.14}
\end{equation}
we have $z<1$ on this interval. With $d=(n\ell-1)/(n-1)$,
\[
 M_0'=6d-2z-(2\ell-1)z'+\frac{3n}{2(n-1)}
 >\frac{7n-8}{2(n-1)}>0.
\]
At the initial point, using $z(2/3)<1$ before combining the terms gives
\begin{align*}
 M_0(2/3)&>3\bar p_2(2/3)+A-\frac13-\frac{n}{2(n-1)}\\
 &=\frac16+\frac{(6-n)(4n^2+3n-4)}{6(n-1)(n+2)(n+4)}\ge\frac16.
\end{align*}
Thus $M_0>0$, and \eqref{eq:4.2} proves the result.
\end{proof}

\section{The spectral gap and the upper range}
\label{sec:upper}

For $n\ge7$ and $2/3\le\ell\le\alpha$, we estimate $|W|^2$ using
the self-adjoint action of $G$ on trace-free symmetric cubics. A
min--max argument gives a spectral gap and reduces the projection
bound to an affine expression in the transverse variance. We then
complete Proposition~\ref{prop:sextic}, including the separate
three-dimensional identity.

\subsection{A spectral gap on trace-free symmetric cubics}

The Rayleigh quotient of the axial cubic and a bound on its orthogonal
complement determine the spectral gap.

For a symmetric two-tensor $H$, define its action on $\Sym^3_0(\R^n)$ by
\[
 T_HU=\frac13h_3\bigl(H_{ia}U_{ajk}+H_{ja}U_{iak}+H_{ka}U_{ija}\bigr).
\]
For trace-free symmetric cubics $U_1,U_2$, orthogonal projection and
complete symmetry give
\[
 \langle T_HU_1,U_2\rangle=H_{ia}(U_1)_{ajk}(U_2)_{ijk}.
\]
Thus $T_H$ is self-adjoint, and $H_1\le H_2$ implies $T_{H_1}\le T_{H_2}$.
Moreover, $T_I=I$, and the definitions of $N$ and $W$ give
\begin{equation}
 \langle T_G\sigma,\sigma\rangle=p_2,
 \qquad W=3(T_G-p_2I)\sigma.
 \tag{5.1}\label{eq:5.1}
\end{equation}
Consequently, $|W|^2/9$ is the variance of $T_G$ in the unit tensor
$\sigma$.

\begin{lemma}\label{lem:gap}
Suppose $S=1$, $\ell>1/2$, and set
\begin{equation}
 a=\alpha\ell+\frac{2(1-\ell)}{(n-1)(n+2)},
 \qquad b=1-\ell+\frac{2(n+1)}{3(n+2)}(2\ell-1).
 \tag{5.2}\label{eq:5.2}
\end{equation}
If $\ell\ge(n+5)/(2(n+2))$, then
\begin{equation}
 |W|^2\ge9(a-p_2)(p_2-b).
 \tag{5.3}\label{eq:5.3}
\end{equation}
\end{lemma}

\begin{proof}
Choose a Gram eigenframe with $g_1=\ell$, and put $E=e_1\otimes e_1$.
Let $Z_1$ be the unit trace-free symmetric cubic with independent components
\[
 (Z_1)_{111}=\sqrt{\frac{n-1}{n+2}},\qquad
 (Z_1)_{1\beta\gamma}=-\frac{\delta_{\beta\gamma}}{\sqrt{(n-1)(n+2)}},
 \qquad (Z_1)_{11\beta}=(Z_1)_{\beta\gamma\delta}=0.
\]
Its Rayleigh quotient is
\[
 \langle T_GZ_1,Z_1\rangle
 =\alpha\ell+\frac{2}{(n-1)(n+2)}\sum_{i>1}g_i=a.
\]
Hence the largest eigenvalue of $T_G$ is at least $a$.

For $U\perp Z_1$, the axial scalar component in the decomposition
\eqref{eq:2.1} vanishes. Therefore
\[
 \langle T_EU,U\rangle=2|u|^2+|C|^2
 \le\frac{2(n+1)}{3(n+2)}|U|^2.
\]
Since the transverse Gram eigenvalues have sum $1-\ell$,
\[
 G\le(1-\ell)I+(2\ell-1)E.
\]
The order property gives $\langle T_GU,U\rangle\le b|U|^2$ on the
codimension-one space $Z_1^\perp$. By the min--max principle, the second
and all subsequent eigenvalues of $T_G$ are at most $b$. Finally,
\[
 a-b=\frac{(n-2)\{2(n+2)\ell-n-5\}}{3(n-1)(n+2)}\ge0.
\]
Thus the spectrum avoids $(b,a)$, including the case $a=b$, and
$(T_G-aI)(T_G-bI)\ge0$. Pairing with $\sigma$ and using
\eqref{eq:5.1} yields
\[
 |T_G\sigma|^2-(a+b)p_2+ab\ge0,
 \qquad \frac{|W|^2}{9}=|T_G\sigma|^2-p_2^2,
\]
which proves \eqref{eq:5.3}. At the Calabi cubic, $T_G\sigma=p_2\sigma$ and $a=p_2$, so equality
holds in \eqref{eq:5.3}.
\end{proof}

\subsection{The upper spectral range}

We estimate the constant and linear coefficients of the resulting
affine bound separately.

\begin{proposition}\label{prop:upper}
If $n\ge7$, $S=1$, and $2/3\le\ell\le\alpha$, then
$\mathcal B_{\tau_n}\ge0$.
\end{proposition}

\begin{proof}
The spectral-gap hypothesis holds because
\[
 \frac23-\frac{n+5}{2(n+2)}=\frac{n-7}{6(n+2)}\ge0.
\]
Use the transverse moments and $\eta$ from Section~\ref{sec:intermediate}.
Substitution of \eqref{eq:5.3} into \eqref{eq:3.4} cancels the quadratic
moment:
\[
 \frac32p_2^2+\frac16|W|^2
 \ge\frac32(a+b)p_2-\frac32ab.
\]
Together with \eqref{eq:4.1}, this gives the affine bound
\begin{equation}
 \mathcal B_{\tau_n}\ge F_n(\ell)+M_n(\ell)\eta,
 \tag{5.4}\label{eq:5.4}
\end{equation}
where
\[
 F_n=b_0+\frac32(a-\bar p_2)(\bar p_2-b),
 \qquad
 M_n=A-(2\ell-1)z+\frac32(a+b)-\frac{3n}{2(n-1)}(1-\ell).
\]
We estimate the two coefficients separately.

By \eqref{eq:4.14}, $z<1$ and $z'<0$ for $\ell\ge2/3$.
Differentiating \eqref{eq:5.2} gives
\[
 M_n'=-2z-(2\ell-1)z'
       +\frac32\left(a'+b'+\frac{n}{n-1}\right)
 >\frac{3n^2-4n+4}{2(n-1)(n+2)}>0.
\]
At the left endpoint, use $z(2/3)<1$ before combining the fractions:
\[
 \begin{split}
 M_n(2/3)
 &>A-\frac13+\frac32\{a(2/3)+b(2/3)\}-\frac{n}{2(n-1)}\\
 &=\frac{25n^2-18n+8}{6(n-1)(n+2)(n+4)}>0.
 \end{split}
\]
Thus $M_n>0$ throughout the interval.

Set $\varepsilon=\alpha-\ell$ and $d=(n\ell-1)/(n-1)$.
The identity $a-\bar p_2=d\varepsilon$ gives
\[
 \frac{F_n}{\varepsilon}=\varphi_n+\frac32d(\bar p_2-b)
 \qquad(\varepsilon>0).
\]
Use \eqref{eq:4.7} and $\bar p_2=(1+(n-1)d^2)/n$.
The cubic terms cancel, and retaining the linear denominator gives
\begin{equation}
 \begin{split}
 \frac{F_n}{\varepsilon}
 &=\frac{n(n-2)}{2(n-1)(n+2)}
   -\frac{n^2-6n-8}{2n(n+1)(n+2)}d\\
 &\quad+\left\{
 \frac{3n^3+4n^2-20n+16}{2n(n-2)(n+2)}
 -\frac{3n(n-1)(n+2)(5n+14)}{2(n-2)(n+1)(n+4)\Lambda}
 \right\}d^2.
 \end{split}
 \tag{5.5}\label{eq:5.5}
\end{equation}
All terms extend continuously to $\varepsilon=0$.

Multiply \eqref{eq:5.5} by the positive denominator
\[
 n\Lambda=(n+1)(n+6)+3(n-1)(n+2)d,
\]
and put $P(d)=n\Lambda F_n/\varepsilon$.
The rational subtraction contributes only
\[
 -\frac{3n^2(n-1)(n+2)(5n+14)}{2(n-2)(n+1)(n+4)}d^2.
\]
Multiplying the other coefficients by the two terms of $n\Lambda$ and
collecting equal powers of $d$ gives
\begin{equation}
 \begin{split}
 P(d)
 &=\frac{n^2}{2}\{1-d+d(3d-2)^2\}
   +\frac n2(4-7d+10d^2+21d^3)\\
 &\quad-5+d-2d^2-15d^3
   +\frac{4(4+6d-9d^2)}{3(n+2)}-\frac7{3(n-1)}\\
 &\quad+\frac{4d^2(3d-2)}{n-2}
   +\frac{12d(1-d)^2}{n}+\frac{80d^2}{n+4}.
 \end{split}
 \tag{5.6}\label{eq:5.6}
\end{equation}
On the actual spectral range,
\[
 \frac12\le\frac{2n-3}{3(n-1)}\le d
 \le\frac{(n-2)(n+1)}{(n-1)(n+2)}<1.
\]
It suffices to estimate \eqref{eq:5.6} on $1/2\le d\le1$.
Write
\[
 H_0(d)=1-d+d(3d-2)^2,\quad
 Q_0(d)=4-7d+10d^2+21d^3,\quad
 C_0(d)=-5+d-2d^2-15d^3,
\]
and denote the five fractions in \eqref{eq:5.6} by $R_n(d)$.
Here $H_0\ge0$ and, since $d^2\ge d/2$,
\[
 Q_0\ge4-2d+21d^3>0.
\]
Also $4+6d-9d^2\ge4-3d\ge1$. Only two terms of $R_n$ can be
negative, and
\[
 -\frac7{3(n-1)}> -\frac12,
 \qquad
 \frac{4d^2(3d-2)}{n-2}\ge-\frac2{n-2}\ge-\frac12.
\]
Hence $R_n>-1$. Put $x=d-1/2\ge0$. As $n/2\ge3$,
\[
 \begin{split}
 P(d)&>\frac{n^2}{2}H_0(d)+3Q_0(d)+C_0(d)-1\\
 &=\frac{n^2}{2}H_0(d)+9+44x+100x^2+48x^3\ge9.
 \end{split}
\]
Consequently,
\[
 F_n(\ell)>\frac{9(\alpha-\ell)}{n\Lambda}>0
 \qquad(\ell<\alpha).
\]
The quotient in \eqref{eq:5.5} has a finite limit at $\ell=\alpha$,
so $F_n(\alpha)=0$. Since $\Lambda\le4(n+1)\le6n$,
$9/(n\Lambda)\ge1/n^2$.
The lower bound for $M_n(2/3)$ holds throughout the interval and is
larger than $1/n$: indeed, for $n\ge7$,
\[
 25n^2-18n+8>18n^2,
 \qquad
 6(n-1)(n+2)(n+4)<6n\cdot\frac{3n}{2}\cdot2n=18n^3.
\]
Discarding the excess in both coefficients therefore gives
\[
 \mathcal B_{\tau_n}\ge\frac{\alpha-\ell}{n^2}+\frac\eta n\ge0.
\]
\end{proof}

\subsection{Proof of the algebraic proposition}

We combine the spectral estimates for $n\ge4$ and treat dimension
three by the Ricci decomposition of its algebraic curvature tensor.

\begin{proof}[Proof of Proposition~\ref{prop:sextic}]
For $S=0$ the assertion is immediate. If $n\ge4$, normalize $S=1$.
Proposition~\ref{prop:lower} covers $\ell\le1/2$ and
Proposition~\ref{prop:middle} covers $1/2\le\ell\le2/3$.
For $n=4$, $\alpha=2/3$, so there is no further interval.
For $n=5,6$, Proposition~\ref{prop:small-upper} covers $[2/3,\alpha]$;
for $n\ge7$, this interval is covered by Proposition~\ref{prop:upper}.
Homogeneity restores the factor $S^3$.

In dimension three, retain $S$ without normalization and put
$p_2=\tr(G^2)$, $p_3=\tr(G^3)$.
Form the algebraic curvature tensor given by the right-hand side of
the Gauss equation. Its Ricci contraction is $2I-G$. In a Gram
eigenframe, the three-dimensional Ricci decomposition gives
\[
 K_{ij}=1+g_k-S/2,\qquad\{i,j,k\}=\{1,2,3\}.
\]
The quadratic part of the Gauss tensor has components
\[
 \sum_a(\sigma_{jka}\sigma_{ila}-\sigma_{ika}\sigma_{jla})
 =-[A_i,A_j]_{kl}.
\]
In dimension three, this tensor is diagonal on two-forms in a Ricci
eigenbasis. Each sectional component occurs four times in its
ordered-index squared norm, so
\[
 Q=p_2+4\sum_k(g_k-S/2)^2=5p_2-S^2.
\]
Removing the constant-curvature term in \eqref{eq:2.4} also gives
\[
 \frac12\sum_a|[G,A_a]|^2+|V|^2
 =-\sum_{i<j}(g_k-S/2)(g_i-g_j)^2
 =3p_3-\frac12SQ.
\]
The last equality follows by expanding
\[
 \sum_{i<j}(g_i-g_j)^2=3p_2-S^2,
 \qquad
 \sum_{i<j}g_k(g_i-g_j)^2=4Sp_2-3p_3-S^3.
\]
Consequently,
\[
 \mathcal B_t=\left(\frac{17}{5}-5t\right)Sp_2-3p_3
       +\left(\frac{11}{5}t-\frac45\right)S^3.
\]
The characteristic-polynomial identity gives
\[
 9\det\left(\frac35Sg-G\right)
 =\frac95Sp_2-3p_3-12\left(\frac S5\right)^3.
\]
Matching the $Sp_2$ coefficient selects $t=8/25$, and the $S^3$
coefficient then agrees. This proves \eqref{eq:2.9}.
Lemma~\ref{lem:gram} makes $(3S/5)g-G$ positive semidefinite,
so its determinant is nonnegative.

The estimates also prove the equality statement. For $n\ge4$ and
$\ell<\alpha$, the lower-range bound is strictly positive.
On the intermediate range, $b_0\ge(\alpha-\ell)(1-\ell)/2>0$,
$b_1>0$ and $b_2>0$ make the Bernstein quadratic strictly positive.
For $n=5,6$, the same lower bound for $b_0$, together with $M_0>0$,
gives strict positivity on the upper range; for $n\ge7$ it follows
from the positive Gram-deficit term in Proposition~\ref{prop:upper}.
Thus $S>0$ and $\mathcal B_t=0$ imply $\lambda_{\max}(G)=\alpha S$.
For $n=3$, zero determinant of the positive semidefinite tensor
$(3S/5)g-G$ gives the same conclusion.
\end{proof}

\section{The Calabi axis and the proof of the main theorem}
\label{sec:rigidity}

The integral identity forces $S\equiv0$ or $S\equiv c_n$. In the
nonzero case, the algebraic equality determines the Calabi cubic.
Codazzi makes its axis parallel, and integration of the ambient
equations gives the model immersion and the global covering statement.

\begin{lemma}\label{lem:axis}
Let $e$ be a smooth unit vector field on a Riemannian $n$-manifold,
and put $\xi=e^\flat$. Suppose the cubic tensor
\begin{equation}\tag{6.1}\label{eq:6.1}
 \sigma_{ijk}=\frac1{\sqrt n}
 \bigl((n+2)\xi_i\xi_j\xi_k
       -\xi_i g_{jk}-\xi_j g_{ik}-\xi_k g_{ij}\bigr)
\end{equation}
satisfies Codazzi, so that $\nabla\sigma$ is completely symmetric.
Then $\nabla e=0$ and $\nabla\sigma=0$.
\end{lemma}

\begin{proof}
At a point choose an orthonormal frame with $e_1=e$, and set
$h_{\beta l}=\langle\nabla_{e_l}e,e_\beta\rangle$ for $\beta>1$.
Since $e$ has unit length, differentiating \eqref{eq:6.1} gives
\begin{equation}\tag{6.2}\label{eq:6.2}
 T_{111l}=0,\qquad
 T_{11\beta l}=\frac{n+1}{\sqrt n}h_{\beta l},\qquad
 T_{1\beta\gamma l}=0,
 \quad T_{ijkl}=\sigma_{ijk;l}.
\end{equation}
Complete symmetry implies
$T_{11\beta1}=T_{111\beta}=0$ and
$T_{11\beta\gamma}=T_{1\beta\gamma1}=0$.
Thus every $h_{\beta l}$ vanishes, giving $\nabla e=0$.
Formula~\eqref{eq:6.1} then gives $\nabla\sigma=0$.
\end{proof}

\begin{proof}[Proof of Theorem~\ref{thm:main}]
Choose $t=8/25$ for $n=3$ and $t=\tau_n$ for $n\ge4$; in either
case $t\ge\tau_n$. Lemma~\ref{lem:derivative},
Proposition~\ref{prop:sextic}, and \eqref{eq:2.8} make all three
terms in \eqref{eq:2.7} nonnegative. Their integrals vanish, so each
term vanishes pointwise. In particular,
\[
 (1-S/\cn)S^2=0.
\]
Connectedness gives $S\equiv0$ or $S\equiv\cn$.

If $S=0$, the ambient Gauss formula shows that the real span of $F$
and $dF(TM)$ is a fixed $(n+1)$-plane. This plane is Lagrangian by
the Legendrian condition. The immersion into its great sphere is a
local isometry; compactness makes it a surjective covering, and simple
connectedness of $\Sph^n$ makes the covering one-sheeted. A unitary
transformation sends the plane to the standard real subspace.

Suppose henceforth that $S=\cn$. We also have $\mathcal B_t=0$.
The equality implication in Proposition~\ref{prop:sextic} gives
$\lambda_{\max}(G)=\alpha S$. Lemma~\ref{lem:gram} determines
the cubic, with independent components
\[
 \sigma_{111}=\frac{n-1}{\sqrt n},\qquad
 \sigma_{1\beta\gamma}=-\frac{\delta_{\beta\gamma}}{\sqrt n},
 \qquad \sigma_{11\beta}=\sigma_{\beta\gamma\delta}=0,
\]
after orienting its axis. In particular,
\[
 G=\operatorname{diag}(n-1,2/n,\ldots,2/n),\qquad
 V=\frac{(n-1)(n-2)(n+1)}{n\sqrt n}\,e^\flat.
\]
Since $n\ge3$, the one-form $V$ is nowhere zero, so
$e=V^\sharp/|V|$ defines a smooth global axis. Its cubic is exactly
\eqref{eq:6.1}. Lemma~\ref{lem:axis} now gives $\nabla e=0$ and
$\nabla\sigma=0$.

We identify tangent vectors with their images under $dF$. For $X\perp e$,
the ambient Gauss formula gives
\begin{equation}\tag{6.3}\label{eq:6.3}
 \begin{aligned}
 D_Ze&=-\frac1{\sqrt n}JZ+\sqrt n\,\xi(Z)Je-\xi(Z)F,\\
 D_ZX&=\nabla_ZX-
       \frac{\xi(Z)JX+\langle Z,X\rangle Je}{\sqrt n}
       -\langle Z,X\rangle F.
 \end{aligned}
\end{equation}
On the universal cover, the parallel one-form $\xi$ is exact; write
$\dd t=\xi$. Define
\begin{equation}\tag{6.4}\label{eq:6.4}
 q=e^{-i\sqrt n t}\frac{F-\sqrt n\,Je}{\sqrt{n+1}},\qquad
 \psi=e^{it/\sqrt n}\frac{\sqrt n F+Je}{\sqrt{n+1}}.
\end{equation}
For example,
$D_Z(F-\sqrt n\,Je)=i\sqrt n\,\xi(Z)(F-\sqrt n\,Je)$.
Differentiating both expressions therefore yields
\begin{equation}\tag{6.5}\label{eq:6.5}
 Dq=0,\qquad
 \dd\psi(Z)=\sqrt{\frac{n+1}{n}}\,e^{it/\sqrt n}
             (Z-\xi(Z)e).
\end{equation}
For $X\perp e$, a second differentiation gives
\begin{equation}\tag{6.6}\label{eq:6.6}
 D_Z(\dd\psi(X))
 =\dd\psi(\nabla_ZX)-\frac{n+1}{n}\langle Z,X\rangle\psi.
\end{equation}
Indeed, the phase derivative cancels the term
$-\xi(Z)JX/\sqrt n$ in \eqref{eq:6.3}; the remaining normal terms
combine into a multiple of $\sqrt n F+Je$. These identities hold
for arbitrary $Z$.

Before multiplication by their phase factors, both
$F-\sqrt n\,Je$ and $\sqrt n F+Je$ have squared norm $n+1$.
The real and imaginary parts of their Hermitian inner product vanish, using
$\langle F,Je\rangle=\langle JF,e\rangle=0$.
Their Hermitian products with a transverse vector $X$ vanish as well.
Multiplication by complex phases preserves these zero products.
Thus $q$ and $\psi$ are unit vectors, and $q$ is Hermitian orthogonal
to $\psi$ and $\dd\psi(TM)$.

By \eqref{eq:6.5}--\eqref{eq:6.6}, ambient derivatives of local fields spanning
$\operatorname{span}_{\R}\{\psi,\dd\psi(e^\perp)\}$ remain in this
span. It therefore defines a fixed real $n$-plane on the connected
universal cover. The Legendrian condition gives
\[
 \langle J\psi,\dd\psi(X)\rangle=0,\qquad
 \langle J\dd\psi(X),\dd\psi(Y)\rangle=0
 \quad(X,Y\perp e).
\]
Hence this real plane is Lagrangian in the complex orthogonal
complement of $q$. A fixed unitary transformation sends $q$ to the
first coordinate vector and the plane to $\R^n\subset\C^n$.
After this transformation, we identify $\psi$ with its last $n$
coordinates, so that $\psi$ takes values in
$\Sph^{n-1}\subset\R^n$. Let $\widetilde F$ denote the lift of $F$ to
the universal cover. Solving \eqref{eq:6.4} gives
\[
 \widetilde F(p)=\left(
 \frac{e^{i\sqrt n\,t(p)}}{\sqrt{n+1}},
 \sqrt{\frac n{n+1}}e^{-it(p)/\sqrt n}\psi(p)\right).
\]
Hence the lifted immersion has image in \eqref{eq:1.1}, with parameters
$\theta=t(p)/\sqrt n$ and $x=\psi(p)$.

We determine the global identifications of the model. Its induced
metric is
\[
 n\,\dd\theta^2+\frac n{n+1}g_{\Sph^{n-1}},
\]
so it is an immersion. Two parameters $(\theta,x)$ and $(\varphi,y)$
have the same image exactly when
\[
 e^{in(\varphi-\theta)}=1,\qquad
 y=e^{i(\varphi-\theta)}x.
\]
Since $x,y$ are real unit vectors, $\varphi-\theta\in\pi\mathbb Z$.
Modulo $2\pi$, the only extra identification is
$(\theta,x)\sim(\theta+\pi,-x)$ for even $n$.
The map on the stated compact quotient is injective, hence an
embedding. Call its connected image $C_n$. The original immersion
has image in $C_n$ and is a local diffeomorphism onto it. Its image
is open and, by compactness, closed, so it equals $C_n$.
A proper local diffeomorphism is a covering, and compactness makes
the covering finite.

Finally, the model is minimal Legendrian and attains the pinching value. For \eqref{eq:1.1}, take
$e=F_\theta/\sqrt n$ and unit transverse vectors
$e_\beta=(0,e^{-i\theta}v_\beta)$, where the $v_\beta$ are
orthonormal tangent vectors to $\Sph^{n-1}$. The contact form vanishes:
\[
 \langle JF,F_\theta\rangle=\frac n{n+1}-\frac n{n+1}=0,
 \qquad \langle JF,\dd F(v)\rangle=0.
\]
Differentiation gives
\[
 \langle D_ee,Je\rangle=\frac{n-1}{\sqrt n},\qquad
 \langle D_ee_\beta,Je_\gamma\rangle
 =-\frac{\delta_{\beta\gamma}}{\sqrt n},
\]
while $\sigma_{11\beta}$ and $\sigma_{\beta\gamma\delta}$ vanish.
By symmetry, equivalently,
\[
 B(e,e)=\frac{n-1}{\sqrt n}Je,\qquad
 B(e,e_\beta)=-\frac1{\sqrt n}Je_\beta,\qquad
 B(e_\beta,e_\gamma)=-\frac{\delta_{\beta\gamma}}{\sqrt n}Je.
\]
Consequently,
\[
 \tr B=0,\qquad
 |B|^2=\frac{(n-1)^2+3(n-1)}n=\cn.
\]
\end{proof}
\vspace{2mm}
\noindent\textbf{Acknowledgements.} Yong Luo is supported by the National Natural Science Foundation of China (Grant Nos. 12671062, 12571055). Linlin Sun acknowledges support from the National Natural Science Foundation of China (Grant Nos. 12571055 and 12671062), the Natural Science Foundation of Hunan Province (Grant No. 2026JJ20014), the 111 Project (Grant No. D23017), and the Program for Science and Technology Innovative Research Team in Higher Educational Institutions of Hunan Province, China. Jiabin Yin was supported by the National Natural Science Foundation of China (Grant No. 12201138) and the Natural Science Foundation of Henan Province, China (Grant No. 262300421869).
\vspace{2mm}

\noindent\textbf{AI assistance statement.} The authors used AI models to assist in the preparation of this paper. The authors take full responsibility for the content and final verification of the manuscript.

\end{document}